\documentclass[11pt]{amsart}
\usepackage[colorlinks=true,pagebackref,hyperindex,citecolor=green,linkcolor=red]{hyperref}
\usepackage{amsmath}
\usepackage{amsfonts}
\usepackage{amssymb}
\usepackage{color}
\usepackage[all]{xy}  
\usepackage{enumerate}
\usepackage{verbatim}
\usepackage[top=1in, bottom=1in, left=1.5in, right=1.5in]{geometry}
\usepackage{mathrsfs}
\usepackage{colonequals}
\usepackage{stmaryrd}

\usepackage[textwidth=3.3 cm,textsize=small,shadow]{todonotes}

\theoremstyle{plain}
\newtheorem{theorem}{Theorem}[section]
\newtheorem{theoremx}{Theorem}

\newtheorem{corollary}[theorem]{Corollary}
\newtheorem{lemma}[theorem]{Lemma}

\newtheorem*{theorem*}{Theorem}
\theoremstyle{definition}
\newtheorem{definition}[theorem]{Definition}

\newtheorem{example}[theorem]{Example}

\newtheorem{remark}[theorem]{Remark}
\numberwithin{equation}{subsection}

\newcommand{\m}{\mathfrak{m}}

\newcommand{\NN}{\mathbb{N}}

\newcommand{\sC}{\mathscr{C}}

\newcommand{\Spec}{\operatorname{Spec}}

\newcommand{\Hom}{\operatorname{Hom}}

\newcommand{\p}{\mathfrak{p}}
\newcommand{\q}{\mathfrak{q}}

\newcommand{\n}{\mathfrak{n}}

\newcommand{\e}{\operatorname{e}}

\newcommand{\lr}[1]{{\langle {#1} \rangle}}

\newcommand{\Tr}{\mathrm{Tr}}
\newcommand{\volU}{\mathrm{vol}^{+}}
\newcommand{\volL}{\mathrm{vol}^{-}}
\newcommand{\vol}{\mathrm{vol}}
\newcommand{\bu}{_{\bullet}}

\DeclareMathOperator{\eh}{e}
\DeclareMathOperator{\fsig}{s}
\newcommand{\diffsig}[1]{\operatorname{s}^{\mathrm{diff}}_{#1}}
\DeclareMathOperator{\length}{\ell}

\author{Jack Jeffries}
\address{Centro de Investigaci\'on en Matem\'aticas, M\'exico}
\email{jackjeff@umich.edu}
\author{Ilya Smirnov}
\address{Department of Mathematics, Stockholm University, Sweden}
\email{smirnov@math.su.se}

\begin{document}

\begin{abstract}
For multiplicities arising from a family of ideals
we provide a general approach to transformation rules for a ring extension that is {\'e}tale in codimension one.
Our result can be applied to bound the size of the local {\' e}tale fundamental group of a singularity
in terms of F-signature, recovering a recent result of 
Carvajal-Rojas, Schwede, and Tucker, and differential signature, providing the first
characteristic-free effective bound. 
\end{abstract}

\title{A transformation rule for natural multiplicities}

\maketitle

\section{Introduction}

The F-signature, $\fsig(R)$, of a local ring $(R, \m)$ of positive characteristic is a natural invariant derived from the Frobenius endomorphism. 
There were several notable steps in the theory of F-signature.
Its first traces appear in work of Kunz \cite{KunzReg} and in \cite{SmithVDB} by Smith and Van den Bergh. The definition was 
formally introduced by Huneke and Leuschke in \cite{HLMCM},
while Tucker showed that it is given by a convergent sequence, see \cite{TuckerFSig}.
Aberbach and Leuschke in \cite{AL} showed that F-signature detects strong F-regularity. 
This present paper was motivated by a recent application of F-signature obtained by Carvajal-Rojas, Schwede, and Tucker in
\cite{Javier-Karl-Kevin} where it was shown that the size of the local {\' etale} fundamental group of a strongly F-regular singularity 
is at most $1/\fsig(R)$. 

The motivation for this result came from a question of Koll\'ar \cite{Kollar}, who asked whether the local fundamental group of the germ of a Kawamata log-terminal (KLT)  singularity is finite,
and a partial answer by Xu \cite{Xu}, who showed that the local {\' etale} fundamental group of a  KLT singularity is finite.
The KLT singularities are tightly connected with F-regular singularities in positive characteristic \cite{HaraWatanabe},
thus Carvajal-Rojas, Schwede, and Tucker provided an effective positive characteristic analogue of Xu's result.
Bhatt, Gabber, and Olsson \cite{Bharagav-and-people} showed that Xu's result can be reproved by reduction to positive characteristic, but this proof does not give an explicit bound as in \cite{Javier-Karl-Kevin}. 
Hence it is natural to further search for a direct effective version of Xu's theorem.

A natural tool for such a proof would be F-signature in characteristic zero, a conjectured invariant defined 
as a limiting value of reductions mod $p$. However the existence of such invariant seems to be a very hard problem,
so instead we opt to use differential signature. The differential signature of a local $K$-algebra $R$, denoted $\diffsig{K}(R)$, is a characteristic-free invariant recently introduced by Brenner, {N{\'u}{\~n}ez-Betancourt}, and the first author in \cite{BJNB}. A notable aspect of differential signature is that its nonvanishing is closely related to the KLT property of singularities. However, it remains open whether differential signature is positive for all KLT
singularities.

Differential signature is not the only invariant aiming to fill the gap in characteristic zero; the symmetric signature of Brenner and Caminata \cite{BrennerCaminata, BrennerCaminata2} and normalized volume of a singularity defined by Li \cite{Li, LiLiuXu} are other numerical invariants in characteristic zero that share some important properties with F-signature. 
The key aspect of the proof given by Carvajal-Rojas, Schwede, and Tucker
is a transformation rule for F-signature under a local map $R\to S$ that is module-finite and quasi-{\'e}tale (i.e., \'etale in codimension one). 
Following the approach in \cite{Javier-Karl-Kevin} and keeping these invariants in mind, 
we set a general framework for a transformation rule to hold. We refer the reader to Section~\ref{sec-2} for a precise description of the class of multiplicities specified in the theorem below.

\begin{theoremx}[Theorem~\ref{main-theorem}]
For a class of numerical multiplicities $\mathrm{s}^{\star}(R)$ on local rings $R$ that includes differential signature and F-signature, the following holds:

If $(R,\m,k)\subseteq (S,\n,l)$ are local excellent normal Henselian domains and the inclusion map is finite, splits as $R$-modules, and is quasi-{\' e}tale, then
\[[S:R] \, \mathrm{s}^{\star}(R)=[l:k] \, \mathrm{s}^{\star}(S)\]
where $[S:R]$ denotes the rank of $S$ as an $R$-module.
\end{theoremx}

The class of numerical multiplicities  in the theorem consists of numerical multiplicities computed by a sequence of ideals, and this framework
allows us to treat F-signature and differential signature uniformly. However,
it is yet not clear whether the other two invariants mentioned above can be computed in this way.

As a consequence of this theorem, we can show that differential signature gives characteristic-free bounds on local \'etale fundamental groups.

\begin{theoremx}[Theorem~\ref{cor-fundl-group}]
	Let $(R,\m,k)$ be the Henselization of a local ring essentially of finite type over a separably closed field $K\cong k$. Suppose that $R$ is a normal domain of dimension at least two. If $K$ has positive characteristic, assume also that $R$ is F-pure. 
	
	If the differential signature of $R$ is positive, then for any open subset $U$ of $\Spec(R)$ for which the complement has codimension at least two,  $\pi_1^{\text{\'et}}(U)$ is bounded above by ${1/\diffsig{K}(R) < \infty}$.
\end{theoremx}

We note that while differential signature is known to be positive for large classes of KLT singularities, it remains open whether it is positive for all KLT singularities.

It should be noted that the transformation rule of \cite{Javier-Karl-Kevin} and its application
were further generalized in a different direction in \cite{JavierTorsor}. 
These generalizations 
are based on a different conceptual understanding of F-signature via Frobenius splittings and heavily exploit the Frobenius endomorphism. We do not know whether our approach can be extended to this level of generality.

\section{Natural multiplicities}\label{sec-2}

In this section, we define the properties of asymptotic multiplicities that we use to prove the transformation rule for module-finite maps.

\begin{definition}
\begin{itemize}
	\item In this paper, by a \emph{family of ideals} in a ring $R$, we mean a sequence $I\bu$ of ideals indexed by an infinite subset $\Lambda\subseteq \NN$. We may say that $I\bu$ is \emph{indexed by $\Lambda$}.
	\item An \emph{assignment of a family of ideals}, $A$, consists of a subcategory $\mathcal{R}$ of the category of rings, an infinite subset $\Lambda\subseteq \NN$, and a rule that assigns to each $R\in \mathcal{R}$ a family of ideals $A(R)\bu$ indexed by $\Lambda$.
\end{itemize}
\end{definition}

\begin{example} A few assignments of families of ideals that have been of interest in local algebra include:

\begin{itemize} 	
\item\textbf{Powers of $\m$:} On the category of local rings, the rule $\mathcal{P}$ that assigns to each local ring $(R,\m)$ and integer $n$ the $n$th power of the maximal ideal $\mathcal{P}(R)_n\colonequals\m^n$.

\item\textbf{Frobenius powers of $\m$:} On the category of local rings of characteristic $p>0$, the rule $\mathcal{FP}$ that assigns to each local ring $(R,\m)$ of characteristic $p>0$ and power $q=p^e$ of $p$ the $q$th \emph{Frobenius power} of the maximal ideal, \[\mathcal{FP}(R)_{q}\colonequals\m^{[q]}\colonequals (f^{q} \ | \ f\in \m).\]

\item \textbf{Differential powers of $\m$:} On the category of local $K$-algebras for a field $K$, the rule $\mathcal{DP}$ that assigns to each local ring $(R,\m)$ and integer $n$ the $n$th \emph{differential power} of the maximal ideal \[\mathcal{DP}(R)_n\colonequals\m^{\lr{n}_K}\colonequals\{ f \in R \ | \ D^{n-1}_{R|K} \cdot f \subseteq \m\}.\]
Here, $D^{n-1}_{R|K}$ denotes the $K$-linear \emph{differential operators} on $R$ of order at most $n-1$, in the sense of Grothendieck \cite[\S16.8]{EGAIV}. We refer the reader to \cite[\S2.1]{SurveySP} and \cite[\S3]{BJNB} for an introduction to differential powers of ideals.
\item\textbf{Splitting ideals:} On the category of local rings of characteristic $p>0$, the rule $\mathcal{SI}$ that assigns to each local ring $(R,\m)$ of characteristic $p>0$ and power $q=p^e$ of $p$ the \emph{$q$th splitting ideal}, 
\[\mathcal{SI}(R)_n\colonequals I_q(R)\colonequals\{ f \in R \ | \ \sC^{q}_R \cdot f \subseteq \m \}.\]
Here $\sC^{q}_R$ denotes the set of \emph{Cartier maps of level $e$} or \emph{$p^{-e}$-linear maps} on $R$. If $R$ is reduced, we may identify $\sC^{q}_R$ with $\Hom_R(R^{1/q},R)$. We refer the reader to \cite{Cartiersurv} for an introduction to $p^{-e}$-linear maps. We note here that our choice of indexing (by $q$ rather than $e$) is nonstandard. 
\end{itemize}
\end{example}

\begin{remark}
Throughout the paper, the morphisms in the category of local rings are local homomorphisms (i.e., map the maximal ideal of the source into the maximal ideal of the target). The categories of local rings of characteristic $p>0$, and of local $K$-algebras are taken to be full subcategories of the category of local rings.
\end{remark}

\begin{definition}
	Let $(R,\m)$ be a local ring of dimension $d$. For a sequence $I\bu$ of 
	$\m$-primary ideals indexed by $\Lambda$ 
	we define the \textit{upper volume} and \textit{lower volume} of $I\bu$ to be, respectively,
	\[\volU_R(I\bu)\colonequals\limsup\limits_{i\in \Lambda} \frac{\ell_R(R/I_i)}{i^d} \qquad \text{and} \qquad \volL_R(I\bu)\colonequals\liminf\limits_{i\in \Lambda} \frac{\ell_R(R/I_i)}{i^d}.\]
\end{definition}

\begin{definition} The \emph{upper multiplicity associated to an assignment $A$ of a family of $\m$-primary ideals} is the rule that assigns to each ring $R\in \mathcal{R}$ the number $\e_A^+(R)\colonequals \volU_R(A(R)\bu)$. The lower multiplicity $\e_A^-(R)\colonequals \volL_R(A(R)\bu)$ is defined analogously. If these two numbers coincide, we write $\e_A(R)$ for the common value.
\end{definition}

\begin{example} Many of the multiplicities of interest in local algebra are defined as multiplicities associated to assignments of families of $\m$-primary ideals. In particular, by definition, we have:
	\begin{itemize}
		\item \textbf{Hilbert--Samuel multiplicity:} $\displaystyle \e(R) \colonequals \frac{1}{\dim(R)!}\e_{\mathcal{P}}(R)$.
		\item \textbf{Hilbert--Kunz multiplicity:} $\e_{HK}(R)\colonequals \e_{\mathcal{FP}}(R)$.
		\item \textbf{Differential signature:} $\displaystyle \diffsig{K}(R) \colonequals \frac{1}{\dim(R)!} \e_{\mathcal{DP}}^+(R)$.
		\item \textbf{F-signature:} $\fsig(R)\colonequals \e_{\mathcal{SI}}(R)$.
	\end{itemize}

We note that the definition of F-signature here is not the original definition of Huneke and Leuschke in \cite{HLMCM}
but originates from \cite[Lemma~2.1]{Yao}, \cite[Corollary~2.8]{AE}.
We refer the reader to \cite{CraigSurvey} for a survey on Hilbert--Kunz multiplicity and  F-signature, and to \cite{BJNB} for the definition and basic properties of differential signature.
\end{example}

We now collect some natural properties of multiplicities that the differential signature and F-signature enjoy; verifications are left to Section~\ref{sec-4}.

\begin{definition} Let  $\mathcal{R}$ be a subcategory of the category of rings.
	An ideal $I\subseteq R$ is \emph{characteristic} in $\mathcal{R}$ if for every automorphism $\phi\in \mathrm{Aut}_{\mathcal{R}}(R)$, one has $\phi(I)\subseteq I$.
\end{definition}

\begin{definition}
Let $I\bu$ be a family of ideals in a local ring $(R,\m)$. We say that $I\bu$ is \emph{bounded} if there exists some $a>0$ such that $\m^{an}\subseteq I_n$ for all $n\in \Lambda$.
\end{definition}

\begin{definition}
	Let $\mathcal{R}$ be a subcategory of the category of rings, and $A$ be an assignment of ideals on $\mathcal{R}$. 
Let $\phi \colon R \to S$ be a map in $\mathcal R$.  
We say that $A$ satisfies the \emph{intersection property} for $\phi$ 
if  $A(R)_n=\phi^{-1}(A(S)_n)$ for all $n\in \Lambda$.
\end{definition}

Following standard geometric terminology, we will say that a map $R\to S$ is \emph{quasi-{\' e}tale} if it is {\'e}tale in codimension one, i.e., if for any prime of height one $\q$ of $S$ with contraction $\p=\q \cap R$, the localization $R_{\p} \to S_{\q}$ is \'etale.

Given these definitions, we can state our main transformation rule.

\begin{theorem}[Main Transformation Rule]\label{main-theorem}
	Let $\mathcal{R}$ be a subcategory of the category of local rings. Let $A$ be an assignment of families of ideals on $\mathcal{R}$ such that
	\begin{itemize}
		\item $A(R)_n$ is characteristic in the category $\mathcal{R}$ for all $R\in \mathcal{R}$ and all $n\in \Lambda$,
		\item $A(R)$ is bounded for all $R\in \mathcal{R}$,
		\item $A$ satisfies the intersection property for every morphism in $\mathcal{R}$ that is finite, split, and quasi-\'etale.
	\end{itemize}
Let $(R,\m,k)$ and $(S,\n,l)$ be local excellent normal Henselian domains in $\mathcal{R}$. If there is a finite local map in $\mathcal{R}$ from $R$ to $S$ that splits as $R$-modules and is quasi-{\'e}tale, then
\[[S:R] \e_A^+(R)=[l:k] \e_A^+(S) \qquad \text{and} \qquad [S:R] \e_A^-(R)=[l:k] \e_A^-(S),\]
where $[S:R]$ denotes the rank of $S$ as an $R$-module.
\end{theorem}

We will see in Section~\ref{sec-4} below that differential signature and F-signature satisfy the hypotheses of this theorem. Observe that powers and Frobenius powers of $\m$ in a local ring $(R,\m)$ are characteristic ideals, and that the families of powers $\mathcal{P}(R)_n$ and Frobenius powers $\mathcal{FP}(R)_q\colonequals \m^{[q]}$ are bounded. However, the intersection property for maps that are finite, split, and quasi-{\'e}tale fails for both, e.g., for the inclusion $\mathbb{F}_p\llbracket x^2, xy,y^2 \rrbracket \subseteq \mathbb{F}_p\llbracket x, y \rrbracket$. 

We give the proof of this result, after a collecting a few preliminary lemmas, in the next section.


\section{Proof of main transformation rule}

We first focus on one inequality in the statement of Theorem~\ref{main-theorem}. We first recall the following lemma.

 \begin{lemma}\label{lengthCraig}
 	Let $(R,\m,k)$ be a local domain of dimension $d$ and $N$ be a finitely generated torsion $R$-module. Then, 
for any $a>0$ there is a constant $C$ such that for any $n>0$ and any ideal with $\m^{an} \subseteq I$, one has $\length_R(N/IN)<Cn^{d-1}$.
 \end{lemma}
\begin{proof}
This is essentially \cite[Lemma~3.5]{CraigSurvey}.
\end{proof}

The next statement is one of the two inequalities in Theorem~\ref{main-theorem}.

\begin{lemma} \label{one way}
Let $A$ be a bounded assignment of families of ideals and $(R,\m,k) \subseteq (S,\n,l)$ be a module-finite local inclusion of domains
that satisfies the intersection property for $A$.  
Then
\[
[l:k] \eh^+_A (S) \leq [S:R] \eh^+_A(R)
\text{ and }
[l:k]  \eh^-_A(S) \leq [S:R] \eh^-_A(R).
\]
\end{lemma}
\begin{proof} The proofs will be identical for $\eh_A^+$ and $\eh_A^-$, so we will use the symbol $\eh_A$ 
to denote either.

There is a short exact sequence of $R$-modules $0 \rightarrow F \stackrel{\phi}{\rightarrow} S \rightarrow T \rightarrow 0$, where $T$ is a torsion module and $F$ is a free module of rank $[S:R]$. This gives for each $i$ a short exact sequence of $R$-modules
	\[
	0 \rightarrow F/\phi^{-1}(A(S)_i) \rightarrow S/A(S)_i \rightarrow S/(A(S)_i + \phi(F)) \rightarrow 0.
	\]
By the intersection propery $A(R)_iS = (A(S)_i \cap R)S \subseteq A(S)_i$, so, 
since $\phi$ is $R$-linear, we have
	$\phi \big(A(R)_i F\big) \subseteq A(S)_i$.
Therefore
\[\length_R\big(F/\phi^{-1}(A(S)_i)\big)\leq \length_R\big( F / A(R)_i F \big) = [S:R]\length_R\big(R/A(R)_i\big)\]
and there is a surjection of $R$-modules
	\[T/A(R)_i T \cong S/(A(R)_i S + \phi(F)) \twoheadrightarrow S/( A(S)_i + \phi(F))\]
giving that
$\length_R (S/( A(S)_i + \phi(F)))\leq\length_R (T/A(R)_iT)$.

Furthermore, from the same exact sequence we obtain that
\begin{align*}
[l:k]\length_S (S/A(S)_i)=
\length_R (S/A(S)_i) &= \length_R (F/\phi^{-1}(A(S)_i)) + \length_R (S/(A(S)_i + \phi(F))) \\
&\leq  [S:R]\length_R\big(R/A(R)_i\big) + \length_R (T/A(R)_i T).
\end{align*}
Because $A$ is bounded, it follows from Lemma~\ref{lengthCraig} that $\length_R (T/A(R)_iT) = O(i^{d - 1})$
and the claim follows after passing to the limit. 
\end{proof}

See also \cite[Corollary~4.13]{TuckerFSig} for a result similar to the above for F-signature.

We now pursue the other inequality. We will employ the trace map associated to a finite field extension. To begin, we will recall some basic properties of the trace map.

\begin{lemma}\label{lem:trace}
	Let $(R,\m,k)\subseteq (S,\n,l)$ be a module-finite local inclusion of domains, with fraction fields $K$, $L$, respectively.
	\begin{enumerate}
		\item \label{lem:trace-7} $\Tr_{L/K} \neq 0$ if $L$ is separable over $K$.
		\item \label{lem:trace-6} If $M/L/K$ is a tower of fields, then $\Tr_{M/K}=\Tr_{L/K}\circ \Tr_{M/L}$.
		\item\label{lem:trace-1} If $L/K$ is Galois, then $\Tr_{L/K}(x)=\sum_{\sigma\in \mathrm{Gal}(L/K)}\sigma(x)$.
		\item\label{lem:trace-2} If $R$ and $S$ are normal, then $\Tr_{L/K}(S)\subseteq R$.
		\item\label{lem:trace-3} If $R$ and $S$ are normal, then $\Tr_{L/K}(\n)\subseteq \m$.
		\item\label{lem:trace-5} If $R$ and $S$ are normal, then $\Tr_{L/K}$ generates $\Hom_R(S,R)$ as an $S$-module if and only if the inclusion of $R$ into $S$ is quasi-{\'e}tale.
	\end{enumerate}
\end{lemma}
\begin{proof} Statements (1)--(4) are standard. Statement (5) is \cite[Lemma~2.10]{Javier-Karl-Kevin}. Statement (6) appears in a geometric formulation in \cite[Proposition~4.8]{STFiniteMaps}.
	\end{proof}
The following is evident from the Lemma above.

\begin{corollary}\label{cor:tracesplitting}
Let $(R,\m,k)\subseteq (S,\n,l)$ be a module-finite local inclusion of normal domains, with fraction fields $K$, $L$, respectively. If the inclusion is split and quasi-{\'e}tale, then there is a unit $u\in S$ such that $\Tr_{L/K}(u\cdot -)$ is a splitting of the inclusion.
\end{corollary}

\begin{lemma}\label{lem:image-of-trace}
	Let $(R,\m,k)\subseteq (S,\n,l)$ be a module-finite local inclusion of normal domains, with fraction fields $K$, $L$, respectively. Let $M$ be the Galois closure of $L/K$, and $T$ be the integral closure of $S$ in $M$, and suppose that $T$ is local. Let $\mathcal{R}$ be a full subcategory of the category of local rings that contains $R$,  $S$, and $T$.
	\begin{enumerate}
		\item  If $L/K$ is Galois and $J$ is a characteristic ideal of $S$ in $\mathcal{R}$, then $\Tr_{L/K}(J)\subseteq J\cap R$.
		\item\label{lem:image-of-trace-2} Suppose that there is a unit $u\in T$ such that $\Tr_{M/L}(u\cdot -)$ is a splitting of the inclusion of $S$ into $T$. Let $J\subseteq S$ be an ideal, and suppose that there is a characteristic ideal $J'$ of $T$ in $\mathcal{R}$ such that $J=J'\cap S$. Then $\Tr_{L/K}(J) \subseteq J\cap R$.
	\end{enumerate}
\end{lemma}
\begin{proof}
	\begin{enumerate}
		\item Note that, for $\sigma\in \mathrm{Gal}(L/K)$, the restriction $\sigma|_S$ maps $S$ into $S$ by the hypothesis of normality. The map $\sigma^{-1}|_S$ is an inverse of $\sigma|_S$, so $\sigma|_S$ automorphism of $S$, which necessarily must be local; thus, this map is a morphism in $\mathcal{R}$. For $j\in J$, by Lemma~\ref{lem:trace}(\ref{lem:trace-1}), we have that $\Tr_{L/K}(j)=\sum_{\sigma\in \mathrm{Gal}(L/K)}\sigma(j)$, and since $J$ is characteristic, $g(j)\in J$ for each $g$, so $\Tr_{L/K}(j)\in J$. By Lemma~\ref{lem:trace}(\ref{lem:trace-2}), we also have $\Tr_{L/K}(j)\in R$.
		\item For $j\in J$, by Lemma~\ref{lem:trace}(\ref{lem:trace-6}), we can write $\Tr_{L/K}(j)=\Tr_{L/K}\big(\Tr_{M/L}(uj)\big)=\Tr_{M/K}(u j)\in \Tr_{M/K}(J')$. Applying the first part of the Lemma, we have $\Tr_{L/K}(j)\in \Tr_{M/K}(J')\subseteq J'\cap R = J\cap R$.\qedhere
	\end{enumerate}
\end{proof}

The following lemma is well-known to experts, but we include it for completeness. 

\begin{lemma}\label{lem:GaloisClosure}
	Let $R\subseteq S$ be a module-finite extension of normal domains that is quasi-{\'e}tale. Let $K,L$ be the fraction fields of $R,S$, and $T$ be the integral closure of $S$ in the Galois closure of $L/K$. Then $S\subseteq T$ is quasi-{\'e}tale.
\end{lemma}
\begin{proof}
	By localizing at height one primes and completing, this reduces to the claim that if $R$ (respectively, $S$) is the valuation ring of a local field $K$ (respectively, $L$), and $S$ is \'etale over $R$, then the valuation ring of the Galois closure of $L/K$ is \'etale over $R$. This follows from the fact that the \'etale property for valuation rings of local fields is preserved under taking the compositum of field extensions \cite[Corollary~7.3]{Neukirch}.
\end{proof}

Now we can prove our main theorem.


\

\noindent \textit{Proof of Theorem~\ref{main-theorem}.} Since $R$ and $S$ are excellent and Henselian local rings, the integral closure $T$ of $S$ in the Galois closure of the fraction fields is module-finite over $S$, and hence is a direct product of local rings; since $T$ is clearly a domain, it thus is local. By Lemma~\ref{lem:GaloisClosure} and the intersection property, $A(T)_n \cap S = A(S)_n$. 

By Corollary~\ref{cor:tracesplitting}, the hypotheses of Lemma~\ref{lem:image-of-trace}(\ref{lem:image-of-trace-2}) are satisfied, so $\Tr_{L/K}(A(S)_n)\subseteq A(R)_n$. By Corollary~\ref{cor:tracesplitting} we obtain the other containment $A(R)_n \subseteq \Tr_{L/K}(A(S)_n)$. In particular, we have equalities $\Tr_{L/K}(A(S)_n) = A(S)_n \cap R = A(R)_n$.

There is a short exact sequence of $R$-modules $0 \rightarrow S \xrightarrow{\Phi}  F \rightarrow C \rightarrow 0$, where $C$ is torsion and  $F$ is a free module of rank $[S:R]$. 
By Lemma~\ref{lem:trace}(\ref{lem:trace-5}), every component of $\Phi$ is a multiple of $\Tr_{L/K}$. Consequently, $\Phi(A(S)_n)\subseteq \bigoplus\limits_{[S:R]} \Tr_{L/K}(A(S)_n)$ for all $n$. Thus, for each $n$, there is an exact sequence 
	\[S/A(S)_n \xrightarrow{\Phi} F / \Tr_{L/K}(A(S)_n)F \rightarrow C/\Tr_{L/K}(A(S)_n)C \rightarrow 0.\]
As in the proof Lemma~\ref{one way}, Lemma~\ref{lengthCraig} shows that 
\[\length_R (C/\Tr_{L/K}(A(S)_n)C) = \length_R (C/A(R)_nC) = O(n^{\dim R - 1}),\]
 so, after passing to the limit
and using that $\length_R (S/A(S)_n) = \length_S (S/A(S)_n) [l:k]$, we have that
	\[
	[l:k]\eh_A (S) \geq [S:R] \eh_A(R).
	\]
But the converse is Lemma~\ref{one way} and the assertion follows.\qed

\section{Applications}\label{sec-4}

Now we present applications of our main theorem to local \'etale fundamental groups of singularities. Our treatment 
follows ideas of \cite{Javier-Karl-Kevin}. We refer to \cite{MilneBook} and \cite{Javier-Karl-Kevin} for an introduction to local \'etale fundamental groups.

First, we aim to show that differential signature and F-signature satisfy the hypotheses of Theorem~\ref{main-theorem}. The next two lemmas address two of these conditions.
 
\begin{lemma}
	\begin{enumerate}
		\item Let $(R,\m,k)$ be a local $K$-algebra with pseudocoefficient field $K$. The ideals $\m^{\lr{n}_K}$ are characteristic in the category of local $K$-algebras.
		\item If $R$ is a local ring containing a field of positive characteristic, then the ideals $I_q(R)$ are characteristic in the category of local rings.
	\end{enumerate}

\end{lemma}
\begin{proof}
	Let $\phi\colon R\rightarrow R$ be a $K$-linear local automorphism. First, we claim that $\phi \circ D^n_{R|K} \circ \phi^{-1}=D^n_{R|K}$. For an element $x\in R$, we write $\overline{x}\in D^0_{R|K}$ for the operator of multiplication by $x$. We observe:
	\begin{align*}
	[\phi \circ \delta \circ \phi^{-1}, \overline{x}](y) &=
	(\phi \circ \delta \circ \phi^{-1})(xy) - x \cdot  (\phi \circ \delta \circ \phi^{-1})(y)\\
	&=\big(\phi \circ \delta\big) \big(\phi^{-1}(x) \cdot  \phi^{-1}(y)\big)  - x \cdot  \big(\phi \circ \delta\big) \big( \phi^{-1}(y)\big)\\
	&= \big(\phi \circ \delta \circ \overline{\phi^{-1}(x)}\big)\big(\phi^{-1}(y)\big) - x \cdot  \big(\phi \circ \delta\big)\big( \phi^{-1}(y)\big)\\
	&= \Big(\phi \circ \big([\delta,\overline{\phi^{-1}(x)}]+\overline{\phi^{-1}(x)}\circ \delta\big)\Big)\big(\phi^{-1}(y)\big) - x \cdot \big(\phi \circ \delta\big)\big( \phi^{-1}(y)\big)\\
	&= \big(\phi \circ [\delta,\overline{\phi^{-1}(x)}] \circ \phi^{-1}\big)(y) + \phi\Big(\phi^{-1}(x) \cdot  \delta\big(\phi^{-1}(y)\big)\Big) - x \cdot \big(\phi \circ \delta\big)\big( \phi^{-1}(y)\big)\\
	&= \big(\phi \circ [\delta,\overline{\phi^{-1}(x)}] \circ \phi^{-1}\big)(y),
	\end{align*}
	so $[\phi \circ \delta \circ \phi^{-1}, \overline{x}]=\phi \circ [\delta, \overline{\phi^{-1}(x)}] \circ \phi^{-1}$. The claim then follows from induction on the order of a differential operator. Now, let $f\in \m^{\lr{n}_K}$, and $\delta\in D^{n-1}_{R|K}$. To see that $\phi(f)\in \m^{\lr{n}_K}$, consider $\delta(\phi(f))$. Write $\delta =\phi \circ \nu \circ \phi^{-1}$; we have $\nu\in D^{n-1}_{R|K}$ by the previous claim. Then, $\delta(\phi(f))=(\phi \circ \nu \circ \phi^{-1})(\phi(f))=\phi(\nu(f))\in \m$, since $\nu(f)\in \m$. Thus, $\m^{\lr{n}_K}$ is characteristic.
	
	We argue similarly for $I_q(R)$. Note that any automorphism $\phi$ of $R$ extends uniquely to an automorphism $\widetilde{\phi}$ of $R^{1/q}$ by the rule $\widetilde{\phi}(r^{1/q})=\phi(r)^{1/q}$. We claim that $\phi \circ \Hom_{R}(R^{1/q},R) \circ \widetilde{\phi^{-1}} = \Hom_{R}(R^{1/q},R)$. Indeed, if $\psi\in\Hom_{R}(R^{1/q},R)$, then
	\begin{align*}
	(\phi \circ \psi \circ \widetilde{\phi^{-1}})(r^{1/q} x) &= (\phi \circ \psi)(\widetilde{\phi^{-1}}(r^{1/q}) \cdot \phi^{-1}(x)) =\phi \big( \phi^{-1}(x) \cdot (\psi\circ \widetilde{\phi^{-1}})(r^{1/q})\big)\\
	&=\phi \big( \phi^{-1}(x) \cdot (\psi\circ \widetilde{\phi^{-1}})(r^{1/q})\big) 
	=x\cdot (\phi \circ \psi \circ \widetilde{\phi^{-1}})(r^{1/q}),
	\end{align*}
	so $\phi \circ \psi \circ \widetilde{\phi^{-1}}$ is $R$-linear. Now, let $f\in I_q(R)$, and $\psi\in \Hom_{R}(R^{1/q},R)$. To see that $\phi(f)\in I_q(R)$, consider $\psi(\phi(f)^{1/q}))$. Write $\psi =\phi \circ \eta \circ \widetilde{\phi^{-1}}$; we have $\eta\in \Hom_{R}(R^{1/q},R)$ by the previous claim. Then, \[\psi(\phi(f)^{1/q})=(\phi \circ \eta \circ \widetilde{\phi^{-1}})(\phi(f)^{1/q})=(\phi \circ \eta \circ \widetilde{\phi^{-1}})(\widetilde{\phi}(f^{1/q}))=\phi(\eta(f)) \in \m,\] since $\eta(f)\in \m$. Thus, $I_q(R)$ is characteristic.
\end{proof}

\begin{lemma}\label{lem:contraction} Let $(R,\m,k)\subseteq (S,\n,l)$ be a module-finite local inclusion of normal domains that is split and quasi-{\'e}tale, and let $K\subset R$ be a field.
\begin{enumerate}
\item Differential powers satisfy the intersection property, i.e., $\m^{\lr{n}_K}=R\cap \n^{\lr{n}_K}$ for all $n$.
\item Splitting ideals satisfy the intersection property, i.e., $I_q(R)= R \cap I_q(S)$ for all $q=p^e$.
\end{enumerate}
\end{lemma}
\begin{proof}
The first part is \cite[Proposition~6.14]{BJNB}. The second proceeds along the same lines as \textit{ibid.} By \cite[Lemma~3.5, Corollary~3.7]{STFiniteMaps}, every Cartier map on $R$ extends to $S$, and since $R$ is a direct summand of $S$, every Cartier map on $S$ restricts to $R$. Thus, for $f\in R$,
 we have that $f\notin I_q(R)$ if and only if there is some $\phi\in \sC^q_R$ such that $\phi(f)=1$, which happens if and only if there is a $\psi\in \sC^q_S$ such that $\psi(f)=1$, which in turn is equivalent to $f \notin I_q(S)$.
\end{proof}

We can now apply Theorem~\ref{main-theorem} to differential signature and F-signature.

\begin{theorem}\label{transformation-theorem} Let $(R,\m,k)\subseteq (S,\n,l)$ be a module-finite local inclusion of normal Henselian domains that is split and quasi-{\'e}tale.
\begin{enumerate}
\item If $R$ and $S$ are algebras with a pseudocoefficient field $K$, then we have ${[S:R] \diffsig{K}(R)=[l:k] \diffsig{K}(S)}$.
\item \cite[Theorem~B]{Javier-Karl-Kevin} If $R$ and $S$ have characteristic $p > 0$, then we have $[S:R] \fsig(R)=[l:k] \fsig(S)$.
\end{enumerate}
\end{theorem}
\begin{proof}
Since $\m^n \subseteq \m^{\lr{n}_K}$ for all local $K$-algebras, differential powers are bounded. Similarly, if $(R,\m)$ is local of positive characteristic, and $\mu(\m)$ is the minimal number of generators of $\m$, we have $\m^{\mu(\m)q} \subseteq \m^{[q]} \subseteq I_q(R)$ so the family of ideals $I_q(R)$ is bounded. Then, using the last two lemmas, the result then follows from Theorem~\ref{main-theorem}.
\end{proof}

\begin{remark}
We recall that the Henselization $R^h$ of a local ring $(R, \m)$ can be constructed as follows; see \cite[Th\`eor\'eme~1]{Raynaud}.
We consider the category of pairs $(S, \n)$ where $R \to S$ is  {\' etale}, $\n \cap R = \m$,
and $k(\n) = k(\m)$. 
Then $R^h = \varinjlim S$ with morphisms given by dominance at $\n$.
\end{remark}

\begin{lemma}\label{lem-leq1}
	Let $R$ be a local algebra essentially of finite type over a field $K$, and suppose that $R$ contains a coefficient field. Then $\diffsig{K}(R^h) = \diffsig{K}(R) \leq 1$.
\end{lemma}
\begin{proof}
First, we note that the natural map $R\to R^{h}$ is formally \'etale as a colimit of \'etale algebras
\cite[Lemma~07QN]{stacks-project}, and the modules of differentials of $R$ are finitely presented. Thus the natural map $R^h \otimes_R D^n_{R|K} \to D^n_{R^h | K}$ is an isomorphism for all $n$ by \cite[2.2.10]{Masson}. 
We note also that every ideal primary to the maximal ideal of the Henselization is expanded from $R$, since the map $\widehat{R} \to \widehat{R^h}$ is an isomorphism \cite[Lemma~06LJ]{stacks-project}. 
The same argument as in \cite[Lemma~3.11]{BJNB} then shows that $(\m R^h)^{\lr{n}_K} = \m ^{\lr{n}_K} R^h$. We then have $\length_R(R/\m ^{\lr{n}_K}) = \length_{R^h}({R^h}/(\m R^h)^{\lr{n}_K})$ for all $n$, and the equality $\diffsig{K}(R^h) = \diffsig{K}(R)$ follows. The second inequality is \cite[Proposition~4.20]{BJNB}.
\end{proof}

We apply Theorem~\ref{transformation-theorem} to give an effective bounds for \'etale fundamental groups.

\begin{theorem}[Bound on local \'etale fundamental groups]\label{cor-fundl-group}
	Let $(R,\m,k)$ be the Henselization of a local ring essentially of finite type over a separably closed field $K\cong k$. Suppose that $R$ is a normal domain of dimension at least two. If $K$ has positive characteristic, assume also that $R$ is F-pure. 
	
	If the differential signature of $R$ is positive, then for any open subset $U$ of $\Spec(R)$ for which the complement has codimension at least two,  $\pi_1^{\text{\'et}}(U)$ is bounded above by ${1/\diffsig{K}(R) < \infty}$.
\end{theorem}
\begin{proof}
	This follows from Theorem~\ref{transformation-theorem} along the same lines as \cite[Theorem~A]{Javier-Karl-Kevin}. In particular, as in \emph{ibid.}, it suffices to show that if $(R,\m,k)\subseteq (S,\n,l)$ is module-finite and \'etale on $U$, then $[S:R] \leq 1 / \diffsig{K}(R)$. As the complement of $U$ cannot contain a prime of height one, neither can its preimage in $\Spec(S)$, so such a map is quasi-{\'e}tale.
	
Observe that if $(R,\m,k)\subseteq (S,\n,l)$ is module-finite, then it splits as $R$-modules: if $K$ has characteristic zero, the trace map yields a splitting, while if $R$ has positive characteristic, then $R$ is strongly F-regular, since $R$ is F-pure and has positive differential signature \cite[Theorem~5.17]{BJNB}, and hence is a direct summand of any module-finite extension.

Furthermore, $k\to l$ is an isomorphism: this is algebra-finite, hence module-finite; in characteristic zero $k$ is algebraically closed, and in positive characteristic, by \cite[Lemma~2.15]{Javier-Karl-Kevin} the map is separable, hence an isomorphism. By Theorem~\ref{transformation-theorem}, for such an $S$, we have $[S:R] < \diffsig{K}(S)/\diffsig{K}(R)$.

We claim that $S$ as above can be realized as the Henselization of a local algebra $V$ essentially of finite type over a field $K$ that contains $K$. Write $R=T^h$ with $T$ essentially of finite type over the coefficient field $K$. By the module-finite hypothesis, we can write $S=\sum_{i=1}^t R s_i$ for some finite set of elements $s_i\in S$, with the algebra structure given by the rules $s_i s_j = \sum_{m=1}^t r_{ijm} s_m$. By the construction of Henselization there exists some local ring $A$ such that $A$ is essentially of finite type over $K$, $T\subseteq A \subseteq R$, $A$ is \'etale over $T$, and all of the elements $r_{ijm}$ lie in $A$. Now, let $V=\sum_{i=1}^t A s_i \subseteq S$. We then have $V \otimes_A R \cong S$, and it follows that $S$ is the Henselization of $V$, as claimed. We can then apply Lemma~\ref{lem-leq1} to find that $\diffsig{K}(S) \leq 1$. The theorem then follows.
\end{proof}

\begin{corollary}\label{cor-bound-on-efg}
	Let $X$ be a normal variety of dimension at least two over an algebraically closed field $K$ of characteristic $0$, and $x\in X$ be a closed point. Let $Z\subseteq X$ be a closed subvariety of codimension at least two that contains $x$. Let $R=\mathcal{O}_{X,x}^h$, and let $\tilde{Z}$ be preimage of $Z$ under the composition $\Spec(R) \to \Spec(\mathcal{O}_{X,x}) \to X$.
	Then, $|\pi_1^{\text{\'et}}(\Spec(R)\smallsetminus \tilde{Z})|$ is bounded above by $1/\diffsig{K}(R)$.
	
	 In particular, $|\pi_1^{\text{\'et}}(\Spec^{\circ}(R))|$ is bounded above by $1/\diffsig{K}(R)$.
\end{corollary}

The finiteness statement for \'etale fundamental groups of punctured spectra under the hypothesis that $R$ has KLT singularities over the complex numbers was first shown by Xu \cite{Xu}, and extended by \cite{Bharagav-and-people}.
The effective bound in positive characteristic, with F-signature in place of differential signature, was shown by \cite{Javier-Karl-Kevin}.
The following result is a theorem of Greb, Kebekus, and Peternell \cite[Theorem~1.5]{GKP} under the hypothesis that $X$ has KLT singularities over the complex numbers; the positive characteristic analogue is also known \cite{BCGST}.

\begin{corollary}\label{cor-etale-covers}
Suppose that $X$ is a normal noetherian scheme of finite type over an algebraically closed field $K$ of characteristic $0$. Suppose that for every closed point $x \in X$ 
we have $\diffsig{K}(\mathcal{O}_{X,x}) < \infty$. 
Then for any tower of quasi-{\' e}tale covers of $X$, $X \leftarrow X_1 \leftarrow X_2 \leftarrow X_3 \cdots$,
we must have that $X_i \leftarrow X_{i + 1}$ is {\' e}tale for $i$ sufficiently large.

In particular, there exists a finite, Galois, quasi-{\'e}tale cover $Y\to X$ by a normal scheme $Y$ with the property that any {\'e}tale cover of the regular locus of $Y$ extends to an {\'e}tale cover of $Y$.
\end{corollary}
\begin{proof}
Follows directly from Corollary~\ref{cor-bound-on-efg} and a result of Stibitz \cite[Theorem~1]{Stibitz}.
\end{proof}


\begin{remark}
	\begin{enumerate}
		\item The bound in Theorem~\ref{cor-fundl-group} can be sharp, e.g., for Du Val singularities; see \cite[Theorem~7.1]{BJNB}. However, it is not always sharp, e.g., for determinantal singularities; see  \cite[Theorem~7.1]{BJNB} and \cite[Example~4.0.10]{JavierThesis}.
	
		\item 	It is not known whether the differential signature of the local ring of a Kawamata log-terminal singularity is positive in general, though it is true for direct summands of regular rings and quadric hypersurfaces. If this were true, this would give a characteristic-free proof of the results \cite{Xu, Javier-Karl-Kevin} mentioned above.
		
		\item It is conjectured that the analogous statements to Theorem~\ref{transformation-theorem} and Corollary~\ref{cor-fundl-group} hold with another numerical invariant called the \emph{normalized volume} \cite{LiLiuXu}. To our knowledge, this remains open, and Corollary~\ref{cor-fundl-group} gives the only direct effective numerical bounds on local \'etale fundamental groups in characteristic zero. It is worth to note that by \cite{Blum} the normalized volume can be computed as the volume of a single valuation, so it is also a multiplicity associated to a family of ideals. 
		
		\item In the classes of examples $R$ for which differential signature and F-signature are both known in characteristic $p$, the inequality $\diffsig{K}(R) \leq \fsig(R)$ holds, so that the latter gives stronger bounds on the local \'etale fundamental group. We do not know whether this inequality holds in general.
		
		\item There are other interesting asymptotic multiplicities of singularities that are known to satisfy the transformation rule in special cases, for example, the syzygy symmetric signature and the differential symmetric signature of Brenner and Caminata \cite{BrennerCaminata, BrennerCaminata2}.  We do not know a description of these as multiplicities associated to assignments of ideals. It is natural to ask whether these numerical invariants also fit into the hypotheses of Theorem~\ref{main-theorem}.
				
		\item We note that conclusion of Theorem~\ref{cor-fundl-group} holds for a multiplicity satisfying the hypotheses of Theorem~\ref{main-theorem} if the analogue of Lemma~\ref{lem-leq1} holds for this invariant.

	\end{enumerate}
\end{remark}

\section{Examples}

We now give a couple of examples to illustrate that a generalized transformation rule 
\[ [S:R] \, \vol_R(J_\bullet \cap R) = [l:k] \, \vol_S(J_\bullet) \]
for volumes of sequences of ideals $J_\bullet$, which a fortiori satisfies the intersection property, fails in lieu of the other assumptions employed in Theorem~\ref{main-theorem}.

\begin{example} 
	Consider $R = k \llbracket a,b, x^3 + x^2a + xb \rrbracket  \subset S = k \llbracket a,b,x \rrbracket $. Set $J=(a,b,x)$ in $S$. Note that the powers $J^\bullet$ of $J$ are characteristic ideals, but the inclusion map from $R$ to $S$ is not quasi-{\'e}tale.
	
	Since the contraction is integrally closed and 
	\[
	(x^3 + x^2a + xb)^{\lceil n/2 \rceil} = x^{\lceil n/2 \rceil} (x( x + a) + b)^{\lceil n/2 \rceil}
	\subseteq x^{\lceil n/2 \rceil} (x, b)^{\lceil n/2 \rceil} \subseteq (x, a, b)^n, 
	\]
	we obtain that
	\[
	J^n \cap R \supseteq \overline{(a^n, b^n, (x^3 + x^2a + xb)^{\lceil n/2 \rceil})}.
	\]
	
	In fact, this is an equality. Given a polynomial $F(u,v,w)$, the order of $F(a,b,x^3+x^2a+xb)$ as an element of $S$ is equal to its minimal degree under the weighting $|u|=|v|=1, |w|=2$. Then, since any element of $J^n \cap R$ has order at least $n$, any such element must lie in the integral closure of $(a^n,b^n,(x^3 + x^2a + xb)^{\lceil n/2 \rceil})$.	
	Since $(a^2, b^2, x^3 + x^2a + xb)^n \subseteq  \overline{(a^{2n}, b^{2n}, (x^3 + x^2a + xb)^{n})}$, 
	\begin{align*}\vol(J^\bullet \cap R) &=	\lim_{n \to \infty} \frac{\length_R \left (R/\overline{(a^n, b^n, (x^3 + x^2a + xb)^{\lceil n/2 \rceil})}\right )}{n^3} \\
	&= \frac {1}{48}\eh (a^2, b^2, x^3 + x^2a + xb) = \frac {1}{12} \eh(R) = \frac {1}{12}.\end{align*}
On the other hand, $\vol(J^\bullet)=\frac{1}{6}$, and $[S:R]=3$.
\end{example}

\begin{example}
	Consider $R=k \llbracket x^2,xy,y^2 \rrbracket \subset S=k\llbracket x,y\rrbracket$. Take the sequence of ideals $J_n=(x^{4n},y^{4n},x^{2n+1}-y^{2n})$. Note that the inclusion of $R$ into $S$ is split and quasi-{\'e}tale, but the ideals $J_n$ are not characteristic.
	
	First, one has that $\vol(J_\bullet)=8$. To see this, we note that $J_n=(x^{2n+1}-y^{2n},y^{4n},x^{2n-1}y^{2n})$, and that this generating set form a Gr\"obner basis with respect to the reverse lexicographic ordering. Then, the colength of $J_n$ is the colength of $(x^{2n+1},y^{4n},x^{2n-1}y^{2n})$, which is $8n^2$.
	
	Next, we claim that $J_n\cap R=(x^{4n},x^{2n}y^{2n},x^{2n-1}y^{2n+1},y^{4n})$. The containment ``$\supseteq$'' is clear. To see the other, suppose that $F=A x^{4n} + B y^{4n} + C (x^{2n+1}-y^{2n}) \in J_n$ and consists of only even degree terms. We can assume without loss of generality that $F$ contains no monomial in the ideal $(x^{4n},y^{4n})$. We will show that in this case that $C\in (x^{2n-1},y^{2n})\cap R$, which will justify the claim. Indeed, if $C$ has a nonzero monomial $x^i y^j$ with $i+j$ even, then the odd degree monomial $x^{2n+1+i} y^j$ must be canceled by a term of the form $x^{i'} y^{j'} y^{2n} = x^{i'} y^{2n+j'}$ with $i'=2n+1+i$ and $2n+j'=j$, and $x^{i'} y^{j'}$ a supporting monomial of $C$, or else $2n+1+i$ must be at least $4n$. This implies that this monomial lies in the specified ideal. A similar argument covers the case when $i+j$ is odd.
	
	An easy monomial count then yields that $\vol_R(J_\bullet \cap R)=6$.
\end{example}

\section*{Acknowledgements}

The authors are grateful for helpful conversations with Holger Brenner, Javier Carvajal-Rojas, Elo\'isa Grifo, Mel Hochster, Craig Huneke, Luis N\'u\~nez-Betancourt, Karl Schwede, and Axel St\"abler. We thank the referee for many valuable suggestions and corrections. The first author was supported by NSF Grant DMS~\#1606353.

\bibliographystyle{alpha}
\bibliography{References}

\newcommand{\etalchar}[1]{$^{#1}$}
\def\cprime{$'$} \def\cprime{$'$}
\begin{thebibliography}{BCRG{\etalchar{+}}19}

\bibitem[AE05]{AE}
Ian~M. Aberbach and Florian Enescu.
\newblock The structure of {$F$}-pure rings.
\newblock {\em Math. Z.}, 250(4):791--806, 2005.

\bibitem[AL03]{AL}
Ian~M. Aberbach and Graham~J. Leuschke.
\newblock The {$F$}-signature and strong {$F$}-regularity.
\newblock {\em Math. Res. Lett.}, 10(1):51--56, 2003.

\bibitem[BC17]{BrennerCaminata}
Holger Brenner and Alessio Caminata.
\newblock The symmetric signature.
\newblock {\em Comm. Algebra}, 45(9):3730--3756, 2017.

\bibitem[BC19]{BrennerCaminata2}
Holger Brenner and Alessio Caminata.
\newblock Differential symmetric signature in high dimension.
\newblock {\em Proc. Amer. Math. Soc.}, 2019.
\newblock To appear.

\bibitem[BCRG{\etalchar{+}}19]{BCGST}
Bhargav Bhatt, Javier Carvajal-Rojas, Patrick Graf, Karl Schwede, and Kevin
  Tucker.
\newblock \'{E}tale fundamental groups of strongly {$F$}-regular schemes.
\newblock {\em Int. Math. Res. Not. IMRN}, (14):4325--4339, 2019.

\bibitem[BGO17]{Bharagav-and-people}
Bhargav Bhatt, Ofer Gabber, and Martin Olsson.
\newblock Finiteness of \'etale fundamental groups by reduction modulo $ p$.
\newblock {\em arXiv:1705.07303}, 2017.

\bibitem[BJN19]{BJNB}
Holger Brenner, Jack Jeffries, and Luis {N{\'u}{\~n}ez-Betancour}t.
\newblock Quantifying singularities with differential operators.
\newblock {\em Adv. Math.}, 358:106843, 89~pp., 2019.

\bibitem[Blu18]{Blum}
Harold Blum.
\newblock Existence of valuations with smallest normalized volume.
\newblock {\em Compos. Math.}, 154(4):820--849, 2018.

\bibitem[BS13]{Cartiersurv}
Manuel Blickle and Karl Schwede.
\newblock {$p^{-1}$}-linear maps in algebra and geometry.
\newblock In {\em Commutative algebra}, pages 123--205. Springer, New York,
  2013.

\bibitem[CR17]{JavierTorsor}
Javier Carvajal-Rojas.
\newblock Finite torsors over strongly {$F$}-regular singularities.
\newblock {\em arXiv:1710.06887}, 2017.

\bibitem[CR18]{JavierThesis}
Javier~A Carvajal-Rojas.
\newblock Arithmetic aspects of strong {F}-regularity, Ph.D. thesis, The
  University of Utah, 2018.

\bibitem[CRST18]{Javier-Karl-Kevin}
Javier Carvajal-Rojas, Karl Schwede, and Kevin Tucker.
\newblock Fundamental groups of {$F$}-regular singularities via
  {$F$}-signature.
\newblock {\em Ann. Sci. \'{E}c. Norm. Sup\'{e}r. (4)}, 51(4):993--1016, 2018.

\bibitem[DDSG{\etalchar{+}}18]{SurveySP}
Hailong Dao, Alessandro De~Stefani, Elo{\'i}sa Grifo, Craig Huneke, and Luis
  N{\'u}{\~n}ez-Betancourt.
\newblock Symbolic powers of ideals.
\newblock In {\em Singularities and foliations. geometry, topology and
  applications}, volume 222 of {\em Springer Proc. Math. Stat.}, pages
  387--432. Springer, Cham, 2018.

\bibitem[GKP16]{GKP}
Daniel Greb, Stefan Kebekus, and Thomas Peternell.
\newblock \'{E}tale fundamental groups of {K}awamata log terminal spaces, flat
  sheaves, and quotients of abelian varieties.
\newblock {\em Duke Math. J.}, 165(10):1965--2004, 2016.

\bibitem[Gro67]{EGAIV}
A.~Grothendieck.
\newblock \'{E}l{\'e}ments de g{\'e}om{\'e}trie alg{\'e}brique. {IV}. \'{E}tude
  locale des sch{\'e}mas et des morphismes de sch{\'e}mas {IV}.
\newblock {\em Inst. Hautes {\'E}tudes Sci. Publ. Math.}, (32):361, 1967.

\bibitem[HL02]{HLMCM}
Craig Huneke and Graham~J. Leuschke.
\newblock Two theorems about maximal {C}ohen-{M}acaulay modules.
\newblock {\em Math. Ann.}, 324(2):391--404, 2002.

\bibitem[Hun13]{CraigSurvey}
Craig Huneke.
\newblock Hilbert-{K}unz multiplicity and the {F}-signature.
\newblock In {\em Commutative algebra}, pages 485--525. Springer, New York,
  2013.

\bibitem[HW02]{HaraWatanabe}
Nobuo Hara and Kei-Ichi Watanabe.
\newblock F-regular and {F}-pure rings vs. log terminal and log canonical
  singularities.
\newblock {\em J. Algebraic Geom.}, 11(2):363--392, 2002.

\bibitem[Kol11]{Kollar}
J{\'a}nos Koll{\'a}r.
\newblock New examples of terminal and log canonical singularities.
\newblock {\em arXiv:1107.2864}, 2011.

\bibitem[Kun69]{KunzReg}
Ernst Kunz.
\newblock Characterizations of regular local rings for characteristic {$p$}.
\newblock {\em Amer. J. Math.}, 91:772--784, 1969.

\bibitem[Li18]{Li}
Chi Li.
\newblock Minimizing normalized volumes of valuations.
\newblock {\em Math. Z.}, 289(1-2):491--513, 2018.

\bibitem[LLX18]{LiLiuXu}
Chi Li, Yuchen Liu, and Chenyang Xu.
\newblock A guided tour to normalized volume.
\newblock {\em arXiv:1806.07112}, 2018.

\bibitem[Mas91]{Masson}
Gisli Masson.
\newblock {\em Rings of differential operators and \'etale homomorphisms}.
\newblock ProQuest LLC, Ann Arbor, MI, 1991.
\newblock Ph.D. Thesis, Massachusetts Institute of Technology.

\bibitem[Mil80]{MilneBook}
James~S. Milne.
\newblock {\em \'{E}tale cohomology}, volume~33 of {\em Princeton Mathematical
  Series}.
\newblock Princeton University Press, Princeton, N.J., 1980.

\bibitem[Neu99]{Neukirch}
J\"{u}rgen Neukirch.
\newblock {\em Algebraic number theory}, volume 322 of {\em Grundlehren der
  Mathematischen Wissenschaften [Fundamental Principles of Mathematical
  Sciences]}.
\newblock Springer-Verlag, Berlin, 1999.
\newblock Translated from the 1992 German original and with a note by Norbert
  Schappacher, With a foreword by G. Harder.

\bibitem[Ray70]{Raynaud}
Michel Raynaud.
\newblock {\em Anneaux locaux hens\'{e}liens}.
\newblock Lecture Notes in Mathematics, Vol. 169. Springer-Verlag, Berlin-New
  York, 1970.

\bibitem[ST14]{STFiniteMaps}
Karl Schwede and Kevin Tucker.
\newblock On the behavior of test ideals under finite morphisms.
\newblock {\em J. Algebraic Geom.}, 23(3):399--443, 2014.

\bibitem[{Sta}18]{stacks-project}
The {Stacks Project Authors}.
\newblock \textit{Stacks Project}.
\newblock \url{https://stacks.math.columbia.edu}, 2018.

\bibitem[Sti17]{Stibitz}
Charlie Stibitz.
\newblock {\'E}tale covers and local algebraic fundamental groups.
\newblock {\em arXiv preprint arXiv:1707.08611}, 2017.

\bibitem[SVdB97]{SmithVDB}
Karen~E. Smith and Michel Van~den Bergh.
\newblock Simplicity of rings of differential operators in prime
  characteristic.
\newblock {\em Proc. London Math. Soc. (3)}, 75(1):32--62, 1997.

\bibitem[Tuc12]{TuckerFSig}
Kevin Tucker.
\newblock {$F$}-signature exists.
\newblock {\em Invent. Math.}, 190(3):743--765, 2012.

\bibitem[Xu14]{Xu}
Chenyang Xu.
\newblock Finiteness of algebraic fundamental groups.
\newblock {\em Compos. Math.}, 150(3):409--414, 2014.

\bibitem[Yao06]{Yao}
Yongwei Yao.
\newblock Observations on the {$F$}-signature of local rings of characteristic
  {$p$}.
\newblock {\em J. Algebra}, 299(1):198--218, 2006.

\end{thebibliography}

\end{document}